\documentclass[11pt]{amsart}
\usepackage{amsmath}
\usepackage{amssymb}
\usepackage{amsthm}
\usepackage{latexsym}
\usepackage{hyperref}
\usepackage{enumerate}
\usepackage{fontenc}

\newtheorem{thm}{Theorem}[section]
\newtheorem{cor}[thm]{Corollary}
\newtheorem{lem}[thm]{Lemma}
\newtheorem{prop}[thm]{Proposition}

\newtheorem{thmintro}{Theorem}
\newtheorem{conjintro}[thmintro]{Conjecture}

\theoremstyle{definition}
\newtheorem{defn}[thm]{Definition}

\providecommand{\norm}[1]{\left\| #1 \right\|}
\newcommand{\enuma}[1]{\begin{enumerate}[\textup{(}a\textup{)}] {#1} \end{enumerate}}

\newcommand{\mb}{\mathbf}

\newcommand{\mh}{\mathbb}
\newcommand{\mr}{\mathrm}
\newcommand{\mc}{\mathcal}
\newcommand{\mf}{\mathfrak}

\newcommand{\N}{\mathbb N}
\newcommand{\Z}{\mathbb Z}

\newcommand{\R}{\mathbb R}
\newcommand{\C}{\mathbb C}

\newcommand{\KGK}{K \backslash G / K}

\def\top{{\rm top}}

\begin{document}

\title{On the Baum--Connes conjecture with coefficients for linear algebraic groups}
\author{Maarten Solleveld}
\address{IMAPP, Radboud Universiteit Nijmegen, Heyendaalseweg 135, 
6525AJ Nijmegen, the Netherlands}
\email{m.solleveld@science.ru.nl}
\date{\today}
\thanks{
The author is supported by a NWO Vidi grant "A Hecke algebra approach to the 
local Langlands correspondence" (nr. 639.032.528).}
\maketitle

\begin{abstract}
We prove the Baum--Connes conjecture with arbitrary coefficients for some classes of groups: 

(1) Linear algebraic groups over a non-archimedean local field.

(2) Linear algebraic groups over the adeles of a global field $k$, provided that at every
archimedean place of $k$ the associated Lie group is amenable.

(3) All closed subgroups of the above groups. This includes linear algebraic groups
over global fields - with the same condition as in (2).

\texttt{Proof incomplete, problems in Lemma 2.2.a!} 
\end{abstract}

\tableofcontents

\section*{Introduction}

Let $G$ a locally compact group, always assumed to be Hausdorff and second countable. 
The Baum--Connes conjecture (BC) asserts that two K-groups associated to $G$ are naturally 
isomorphic. The right hand side, or analytic side, of the conjecture is the K-theory of 
the reduced $C^*$-algebra $C_r^* (G)$. The left hand side (or topological side) 
$K_*^\top (G)$ is the equivariant K-homology of a universal space for proper $G$-actions. 
Baum and Connes \cite{BaCo} conjectured that the assembly map
\[
\mu : K_*^\top (G) \to K_* (C_r^* (G)) 
\]
is an isomorphism. 

More general versions of the conjecture include coefficient algebras. For a given 
$G$-$C^*$-algebra $B$ one can replace $C_r^* (G)$ by the reduced crossed product $C_r^* (G,B)$.
The topological side is given in \cite[Definition 9.1]{BCH} in terms of Kasparov
KK-theory \cite{Kas}. Namely, if $X$ is a universal space for proper $G$-actions, then
\begin{equation}\label{eq:1.4}
K_*^\top (G,B) = K_*^G (X,B) = \lim_{X' \subset X ,\, X' / G \text{ compact}} KK_*^G (C_0 (X'), B) .
\end{equation}
The assembly map $\mu$ admits a natural generalization to this context \cite[(9.2)]{BCH}.
The following is known as the Baum--Connes conjecture with coefficients (BCC).

\begin{conjintro}\label{conj:1}
Let $G$ a second countable locally compact group and let $B$ be any $G$-$C^*$-algebra.
Then the assembly map
\[
\mu^B : K_*^G (X,B) \to K_* (C_r^* (G,B)) 
\]
is an isomorphism (of abelian groups).
\end{conjintro}

In case Conjecture \ref{conj:1} holds whenever $B$ is algebra the compact operators on 
some separable Hilbert space, we say that $G$ satisfies the twisted Baum--Connes conjecture.

For an introduction to BC(C) for discrete groups we refer to the very readable booklet \cite{Val}.
Major advantages of Conjecture \ref{conj:1} over the ordinary BC (with coefficients $\C$) are that
BCC is inherited by closed subgroups \cite{ChEc2} and by extensions with amenable groups \cite{CEO}.
In view of these properties (which make ample use of Kasparov's KK-theory) we sometimes
assume in our proofs that our coefficient algebras are separable. But this is actually no
restriction, because for exact groups BC with separable coefficients implies BC with arbitrary
coefficients (Corollary \ref{cor:A.2}).

The main goal of this paper is to verify Conjecture \ref{conj:1} for groups of the form 
$\mc G (R)$, where $\mc G$ is a linear algebraic group and $R$ is some topological ring over
which $\mc G$ is defined. Some (but not too many) of these groups are compact or have the 
Haagerup property, see \cite[\S 1.4]{Jul2}.

Conjecture \ref{conj:1} is already known in the following cases:
\begin{itemize}
\item compact groups \cite{Jul},
\item amenable groups and more generally groups with the Haagerup property \cite{HiKa}.
\end{itemize}
On the other hand, several counterexamples to BCC have been worked out, see \cite{HLS}
All these counterexamples involve non-exact groups \cite{BGW}, so it might be more prudent to 
state Conjecture \ref{conj:1} only for exact groups. Fortunately all the groups we encounter
in this paper are exact (that follows from \cite{KiWa}), so this issue need not bother us.

The ordinary Baum--Connes conjecture is known for larger classes of groups than BCC. 
In particular it has been shown for almost connected locally compact groups \cite{CEN},
for linear algebraic groups over $p$-adic fields \cite{CEN} and for reductive groups
over local fields \cite{Laf}. We point out that for some linear algebraic 
groups over local fields of positive characteristic BC was open (but see below). We refer
to \cite{ELN} for detailed investigation of such groups.

Now we come to our main results.

\begin{thmintro}\label{thm:3} 
Let $F$ be a non-archimedean local field and let $\mc G$ be a linear algebraic group
defined over $F$. Endow $\mc G (F)$ with the topology coming from the metric of $F$.
Then $\mc G (F)$ and all its closed subgroups satisfy BCC.
\end{thmintro}

For all groups in Theorem \ref{thm:3}, the injectivity of the assembly map
$\mu^B$ is due to Kasparov and Skandalis \cite{Kas,KaSk}. The surjectivity of the analogous
map in the context of Banach KK-theory was shown by Lafforgue \cite{Laf}. In particular
there exists a certain unconditional completion $\mc S_t (\mc G (F),B)$ of 
$C_c (\mc G (F),B)$ such that
\begin{equation}\label{eq:L1}
\mu^B_{\mc S_t (\mc G (F)} : K_*^{\mc G (F)} (X,B) \to K_* (\mc S_t (\mc G (F),B))
\end{equation}
is a bijection. See Section \ref{sec:1} for a discussion of these methods.

Our new idea is to consider algebras of rapidly decreasing functions from $\mc G (F)$ 
to a coefficient algebra $B$. We require that these functions decrease (in the 
$L^2$-sense) more rapidly than $\ell^n$ for all $n \in \Z$, where 
$\ell : \mc G (F) \to \R$ is a length function coming from the action of $\mc G (F)$ 
on its Bruhat--Tits building $X$. The crucial point is that, with a generalization of 
results of Vign\'eras \cite{Vig}, one can find such a Fr\'echet algebra which is dense 
and holomorphically closed in $C_r^* (\mc G (F),B)$ (Theorem \ref{thm:2.4}), and 
contains many elements of $\mc S_t (\mc G (F),B)$. Together with \eqref{eq:L1} this 
enables us to establish the surjectivity of $\mu^B$. \\

We also consider linear algebraic groups $\mc G$ defined over a global field $k$
(with the discrete topology). Recall that the adelic group $\mc G (\mb A_k)$ contains 
$\mc G (k)$ as a discrete subgroup.

\begin{thmintro}\label{thm:5} (see Theorem \ref{thm:4.3}) \\ 
Suppose that for every infinite place $v$ of $k$ the Lie group $\mc G (k_v)$ is amenable.
Then the groups $\mc G (\mb A_k)$ and $\mc G (k)$, as well as all their closed subgroups,
satisfy the Baum--Connes conjecture with arbitrary coefficients.
\end{thmintro}

Of course there are a lot of interesting closed subgroups of $\mc G (\mb A_k)$ or $\mc G (k)$,
far too many to list here. Suffice it to refer to \cite{PlRa} and the references therein.

It would be very nice if our method to prove Theorem \ref{thm:3} could be adjusted to a 
real reductive algebraic group $G$. We tried this, but so far it did not work out.
One problem is that rapid decay in an $L^2$-sense does in general not imply rapid decay 
in an $L^\infty$-sense. This makes it difficult to fit an algebra of rapidly decreasing
functions on $G$ in a unconditional completion (in the sense of \cite{Laf}).

As far as arbitrary algebraic groups over $\R$ are involved, our current methods do suffice
to show that $\mc G (\mb A_k)$ satisfies the twisted Baum--Connes conjecture 
(Theorem \ref{thm:4.5}). For this no condition at the archimedean places of $k$
(as in Theorem \ref{thm:5}) is needed.

Finally, we mention one well-known consequence of the surjectivity of the assembly map. 
Kadison and Kaplansky conjectured that, for a torsion--free discrete group $\Gamma$, the 
reduced $C^*$-algebra $C_r^* (\Gamma)$ contains no non-trivial idempotents. As explained
in \cite[\S 7]{BCH} and \cite[\S 6.3]{Val}, this can be deduced from the surjectivity of
\[
\mu : K_0^\top (\Gamma) \to K_0 (C_r^* (\Gamma)) .
\]
From Theorems \ref{thm:3} and \ref{thm:5} we get:

\begin{thmintro}\label{thm:6}
Let $G$ be a group as in Theorem \ref{thm:3} or Theorem \ref{thm:5}.
Let $\Gamma$ be a torsion-free subgroup of $G$ which is discrete in the subspace topology. 
Then $C_r^* (\Gamma)$ contains no idempotents other than 0 and 1.
\end{thmintro}

{\bf Acknowledgments}. We thank Kang Li, Siegfried Echterhoff and Vincent Lafforgue for their 
helpful comments and discussions.

\section{The methods of Kasparov and Lafforgue}
\label{sec:1}

In this section we recall important previous results about the Baum--Connes conjecture
for groups acting on suitable metric spaces. 

Let $G$ be a locally compact group, always tacitly assumed to be Hausdorff and second countable. 
We first suppose that $G$ acts properly and isometrically on an affine building $X$ in the sense of 
\cite{BrTi1,Tit}. These assumptions include that the action is continuous (as usual for topological 
group actions).

Bruhat and Tits showed that $X$ is a CAT(0)-space \cite[3.2.1]{BrTi1}, that it has unique 
geodesics \cite[2.5.13]{BrTi1} and that every compact subgroup of $G$ fixes a point of 
this affine building \cite[Proposition 3.2.4]{BrTi1}. By \cite[Proposition 1.8]{BCH} 
these properties guarantee that $X$ is a universal space for proper $G$-actions. In
particular the domain of the Baum--Connes assembly map becomes $K_*^\top (G) = K_*^G (X)$.

In this section (but not after that) we also consider complete simply connected Riemannian
manifolds with nonpositive sectional curvature which is bounded from below and has
bounded covariant derivative. When a locally compact group $G$ acts properly and
isometrically on such a space $X$, the same arguments as for affine buildings ensure
that $X$ is a universal space for proper $G$-actions. Then \eqref{eq:1.4} is again valid.
Typical examples are symmetric spaces associated to reductive Lie groups.

Another kind of metric spaces to which the results of this section apply are called "bolic" 
\cite{KaSk2}. More precisely, we fix $\delta \in \R_{>0}$ and we let $(X,d)$ be a metric
space such that:
\begin{itemize}
\item It is uniformly locally finite.
\item $(X,d)$ is weakly $\delta$-geodesic \cite[Definition 2.1]{KaSk2}.
\item $(X,d)$ satisfies \cite[(B2')]{KaSk2} for the given $\delta$ and satisfies
condition \cite[(B1)]{KaSk2} for all $\delta' > 0$.
\end{itemize}
By the local finiteness, $X$ is discrete as a topological space. 
If a locally compact group $G$ acts properly and isometrically on $X$, then $X$ cannot 
be a universal example for proper $G$-actions (unless $G$ is compact), because it
is discrete. 

From now on we assume that $G$ acts properly and isometrically on a space $X$ of one of
the above three kinds. With the dual Dirac method Kasparov and Skandalis \cite{KaSk} 
constructed an extremely useful element $\gamma \in KK^G_0 (\C,\C)$.
Via the descent map 
\[
KK^G_0 (\C,\C) \to KK_0 (C_r^* (G),C_r^* (G))
\]
and the product in KK-theory, $\gamma$ gives rise to an endomorphism of $K_* (C_r^* (G))$.
More generally, for every $\sigma$-unital $G$-$C^*$-algebra $B$, $\gamma$ determines an 
element of $\mr{End} \big( K_* (C_r^* (G,B)) \big)$ \cite[\S 3.12]{Kas}.

\begin{thm}\label{thm:1.1} \textup{\cite{Kas,KaSk,KaSk2}} \\
Let $G$ and $X$ be as above. 
\enuma{
\item $\gamma$ is idempotent in $KK_0^G (\C,\C)$.
\item For every $\sigma$-unital $G$-$C^*$-algebra $B$, the assembly map
\[
\mu^B : K_*^G (X,B) \to K_* (C_r^* (G,B))
\]
is injective and has image $\gamma \cdot K_* (C_r^* (G,B))$.
}
\end{thm}

Thus BC for $G$ with coefficients $B$ becomes equivalent to: 
\begin{equation}\label{eq:1.1}
\text{the image of } \gamma \text{ in } \mr{End} \big( K_* (C_r^* (G,B)) \big)
\text{ is the identity.}
\end{equation}
Obviously \eqref{eq:1.1} would be implied by $\gamma = 1 \in KK_0^G (\C,\C)$. 
Although that statement has indeed been proven for some groups acting on 
CAT(0)-spaces, it is known to be false for many others.

This is where Lafforgue's work \cite{Laf} comes in. He developed a KK-theory for 
Banach algebras, which admits a natural transformation from Kasparov's KK-theory.
The advantage is that the image of $\gamma$ in $KK_{Ban}^G (\C,\C)$ can be 1 even
if $\gamma \neq 1$ in $KK_0^G (\C,\C)$. On the other hand, it is not known whether
there exists a natural descent map from $KK_{Ban}^G (\C,\C)$ to 
$KK_{Ban}(C_r^* (G),C_r^* (G))$. For such a descent, one rather has to replace 
$C_r^* (G)$ by suitable Banach algebra completions of $C_c (G)$.

\begin{defn}\label{def:1.2}
A norm on $C_c (G)$ (or an a completion thereof) is unconditional if every
$f \in C_c (G)$ has the same norm as its absolute value $|f|$. 

A Banach algebra $\mc A (G)$ containing $C_c (G)$ as dense subalgebra is said to be
an unconditional completion (for $G$) if its norm is unconditional.
\end{defn}

For every unconditional completion $\mc A (G)$ and every $G$-$C^*$-algebra $B$, there
exists a version $\mc A (G,B)$ of the crossed product of $B$ with $G$. Furthermore
Lafforgue exhibited a Banach algebra version
\[
\mu_{\mc A (G)}^B : K_*^\top (G,B) \to K_* (\mc A (G,B))
\]
of the assembly map.

\begin{thm}\label{thm:1.3} \textup{\cite{Laf}} \\
Suppose that a locally compact second countable group $G$ acts properly and 
isometrically on a space $X$ of one of the above three kinds. 
For every $G$-$C^*$-algebra $B$, and every unconditional completion $\mc A (G)$:
\enuma{
\item The image of $\gamma$ in $\mr{End}\big( K_* (\mc A (G,B)) \big)$ is 1.
\item $\mu_{\mc A (G)}^B : K_*^\top (G,B) \to K_* (\mc A (G,B))$ is a bijection.
}
\end{thm}

The archetypical example of an unconditional completion is $L^1 (G)$. In that case 
Theorem \ref{thm:1.3} says that
\begin{equation}\label{eq:1.2}
\mu_{L^ 1 (G)}^B : K_*^\top (G,B) \to K_* (L^1 (G,B))
\end{equation}
is an isomorphism for all $G,B$ as above. In general the bijectivity of \eqref{eq:1.2}
is known as the Bost conjecture for $G$ (with coefficients $B$). 

The difference between $K_* (C_r^* (G,B))$ and $K_* (\mc A (G,B))$ is an analytic issue, 
which is our main concern in this paper.

\begin{prop}\label{prop:1.4} \textup{\cite[Proposition 1.6.4]{Laf}} \\
In the setting of Theorem \ref{thm:1.3}, suppose that 
\[
\norm{f}_{\mc A (G)} = \norm{f^*}_{\mc A (G)} \geq \norm{f}_{C_r^* (G)}
\qquad \forall f \in C_c (G) .
\]
Then $C_c (G,B) \to C_r^* (G,B)$ extends to a Banach algebra homomorphism 
$\mc A (G,B) \to C_r^* (G,B)$ and (when $B$ is $\sigma$-unital) the induced map
\[
K_* (\mc A (G,B)) \to K_* (C_r^* (G,B))
\]
commutes with multiplication by $\gamma$.
\end{prop}

Proposition \ref{prop:1.4} and Theorem \ref{thm:1.1} show that in $K_* (C_r^* (G,B))$
the images of $\mu^B$, of $\gamma$ and of $K_* (\mc A (G,B))$ coincide.
That leads to a criterion for BC with coefficients for $G$:

\begin{cor}\label{cor:1.5}
In the setting of Theorem \ref{thm:1.3}, suppose that $B$ is $\sigma$-unital and that 
for every class $p \in K_* (C_r^* (G,B))$ there exists an unconditional completion 
$\mc A (G)$ such that:
\begin{itemize}
\item $\norm{f}_{\mc A (G)} = \norm{f^*}_{\mc A (G)} \geq \norm{f}_{C_r^* (G)}
\qquad \forall f \in C_c (G)$,
\item $p$ lies in the image of $K_* (\mc A (G,B)) \to K_* (C_r^* (G,B))$.
\end{itemize}
Then $\mu^B : K_*^\top (G,B) \to K_* (C_r^* (G,B))$ is a bijection.
\end{cor}

Corollary \ref{cor:1.5} applies in particular when
\begin{equation}\label{eq:1.3}
K_* (\mc A (G,B)) \to K_* (C_r^* (G,B))
\end{equation}
can be proven to be surjective. In that case the comparison with $K_*^G (X,B)$ 
shows that \eqref{eq:1.3} is bijective. 

It is not so easy to establish directly that \eqref{eq:1.3} is surjective. When
$\mc A (G,B)$ would be closed under the holomorphic functional calculus of
$C_r^* (G,B)$, that would follow from the density theorem in K-theory 
\cite[Th\'eor\`eme A.2.1]{Bos}. Unfortunately, that seems to be rare for general $B$.

Let $\mc G$ be a reductive algebraic group over a local field $F$ and endow
$G = \mc G (F)$ with the topology coming from the metric of $F$. Lafforgue \cite{Laf} 
constructed completions of $C_c (G,B)$ with several relevant properties. Let 
$\Xi : G \to \R$ be Harish-Chandra's spherical function and let $\ell : G \to \R_{\geq 0}$ 
be a length function associated to the action of $G$ on either its symmetric space
($F$ archimedean) or its Bruhat--Tits building ($F$ non-archimedean).

For $t \in \R$ we define an unconditional norm on $C_c (G,B)$ by
\[
\norm{f}_{\mc S_t (G,B)} = \sup_{g \in G} \norm{f(g)}_B 
\Xi (g)^{-1} (1 + \ell (g))^t .
\]
Let $\mc S_t (G,B)$ be the completion of $C_c (G,B)$ with respect to the above norm
and abbreviate $\mc S_t (G) = \mc S_t (G,\C)$.

\begin{prop}\label{prop:3.5}
There exists $r_G \in \N$ such that for all $t > r_G$:
\enuma{
\item $\mc S_t (G)$ is an unconditional completion of $C_c (G)$;
\item $\mc S_t (G,B)$ is a Banach algebra;
\item $\mc S_t (G,B)$ is contained in $C_r^* (G,B)$ and the inclusion map 
is continuous.
}
\end{prop}
\begin{proof}
(a) According to \cite[Lemme II.1.5]{Wal} (for $F$ non-archimedean), \cite[p. 279]{HC}
(for $F$ archimedean, $\mc G$ semisimple) and \cite[Lemma 27]{Vig} ($F$ archimedean)
there exists $r_G \in \N$ such that 
\[
\int_G \Xi (g)^2 (1 + \ell (g))^{-t} \textup{d}\mu (g) < \infty \quad \text{for all } t > r_G .
\]
Now \cite[Proposition 4.4.4]{Laf} says that $\mc S_t (G)$ is a Banach algebra for $t > r_G$.\\
(b) This follows from part (a) and \cite[Proposition 1.5.1]{Laf}.\\
(c) See \cite[Propositions 4.5.2 and 4.8.2]{Laf}.
\end{proof}

For large $t$, Corollary \ref{cor:1.5} applies to the algebra $\mc S_t (G)$.

\begin{thm}\label{thm:1.6} \textup{\cite{Laf}} \\
Let $\mc G$ be a reductive algebraic group defined over a local field $F$ and write 
$G = \mc G (F)$. 

For $t > r_G$, $\mc S_t (G)$ is an unconditional completion of $C_c (G)$ which is 
holomorphically closed in $C_r^* (G)$. As a consequence, the Baum--Connes conjecture 
(with trivial coefficients) holds for $G$.
\end{thm}

The properties of reductive groups which Lafforgue uses \cite[\S 4.1]{Laf} are quite
specific, they are not available for most other groups. Furthermore Theorem \ref{thm:1.6} 
is not known with nontrivial coefficient algebras. In fact, in \cite[p. 93]{Laf} 
some obstructions are mentioned.

\section{Spaces of rapidly decreasing functions}
\label{sec:2}

The goal of this paragraph is to make full use of results of Vign\'eras, which produce
holomorphically closed subalgebras of $C^*$-algebras.

Let $G$ be locally compact Hausdorff group with a Haar measure $\mu$. Let $B$ be any 
$G$-$C^*$-algebra. Recall that $C_c (G,B)$ acts on the Hilbert $C^*$-module $L^2 (G,B)$, 
by a combination of the convolution product of $G$ and the product of $B$.
The reduced crossed product $C_r^* (G,B)$ is the closure of $C_c (G,B)$ with respect
to the operator norm from $\mc B (L^2 (G,B))$. For $a \in C_r^* (G,B)$, we denote
the corresponding bounded operator on $L^2 (G,B)$ by $\lambda (a)$.

Let $\ell : G \to \R_{\geq 0}$ be a Borel-measurable length function with 
\begin{equation}\label{eq:2.1}
\ell (g) = \ell (g^{-1}) \quad \text{for all } g \in G . 
\end{equation}
We note that pointwise multiplication by $\sigma := 1 + \ell$ is an unbounded operator 
on $L^2 (G,B)$. For $A \in \mc B (L^2 (G,B))$ we define a (possibly unbounded) operator
\[
D(A) : v \mapsto \sigma A(v) - A (\sigma v) . 
\]
Notice that $D(A) = [\sigma,A]$, so that $D$ is a derivation.
Consider the vector space 
\[
V_\ell^\infty (G,B) = \{ a \in C_r^* (G,B) : D^n (\lambda (a)) \in
\mc B (L^2 (G,B)) \; \forall n \in \Z_{\geq 0} \} ,
\]
with the topology given by the seminorms 
\[                                     
a \mapsto \norm{D^n (\lambda (a))}_{\mc B (L^2 (G,B))} \quad n \in \Z_{\geq 0} .
\]
We will now formulate a version of the results of \cite[\S 7]{Vig} for $V_\ell^\infty (G,B)$.
We note that, although Vign\'eras works exclusively with coefficients $\C$, her arguments
are equally valid with other coefficient $G$-$C^*$-algebras.

\begin{thm}\label{thm:2.4}
Let $G$ be a locally compact group and let $B$ be a $G$-$C^*$-algebra.
\enuma{
\item $V_\ell^\infty (G,B)$ is a Fr\'echet algebra containing $C_c (G,B)$.
\item The inclusion $V_\ell^\infty (G,B) \to C_r^* (G,B)$ is continuous, with dense image.
\item The set of invertible elements in the unitization $V_\ell^\infty (G,B)^+$ is open,
and inversion is a continuous map from this set to itself.
\item An element of $V_\ell^\infty (G,B)^+$ is invertible if and only if its image in
$C_r^* (G,B)^+$ is invertible.
}
\end{thm}
\begin{proof}
(a) By \cite[Lemma 15]{Vig} $V_\ell^\infty (G,B)$ is a Fr\'echet space. Since $D$ is 
a derivation, $V_\ell^\infty (G,B)$ is closed under the multiplication of $C_r^* (G,B)$
and multiplication is jointly continuous for the topology of $V_\ell^\infty (G,B)$.

Suppose that $a \in C_c (G,B)$. For $v \in L^2 (G,B)$ and $g' \in G$ we compute
\begin{align}
\nonumber (D (\lambda (a)) v)(g') & = \sigma (g') \int_G a(g) g (v (g^{-1} g')) 
\textup{d}\mu (g) - \int_G a(g) g (v (g^{-1} g')) \sigma (g^{-1}g') \textup{d}\mu (g) \\
\label{eq:2.2} & = \int_G (\sigma (g') - 
\sigma (g^{-1} g')) a(g) g (v (g^{-1} g')) \textup{d}\mu (g) 
\end{align}
As $\ell$ is a length function and by \eqref{eq:2.1}:
\[
|\sigma (g') - \sigma (g^{-1} g')| = | \ell (g') - \ell (g^{-1} g')| \leq
\ell (g^{-1}) = \ell (g)
\]
By \cite[Theorem 1.2.11]{Sch} $\ell$ is bounded on the support of $a$, say by $C_a$. 
Then \eqref{eq:2.2} entails 
\[
\norm{D(\lambda (a))(v)}_{L^2 (G,B)} \leq C_a \norm{a * v}_{L^2 (G,B)} \leq
C_a \norm{\lambda (a)}_{\mc B (L^2 (G,B))} \norm{v}_{L^2 (G,B)} .
\]
With induction we see that $D^n (\lambda (a)) \in \mc B (L^2 (G,B))$ for all 
$n \in \Z_{\geq 0}$, so $C_c (G,B)$ is contained in $V_\ell^\infty (G,B)$. \\
(b) As $\norm{D^0 (\lambda (a))}_{\mc B (L^2 (G,B))} = \norm{a}_{C_r^* (G,B)}$, the inclusion
is continuous. Its image contains $C_c (G,B)$, so is dense in $C_r^* (G,B)$.\\
(c) Any invertible element of $V_\ell^\infty (G,B)^+$ is of the form 
$z + a$ with $z \in \C^\times$ and $a \in V_\ell^\infty (G,B)$. As multiplication by
$z^{-1} \in \C^\times$ is certainly continuous, it suffices to consider the subset
$1 + V_\ell^\infty (G,B)$ of $V_\ell^\infty (G,B)^+$. 

For $a \in V_\ell^\infty (G,B)$ with $\norm{a}_{C_r^* (G,B)} < 1$, \cite[Lemma 16]{Vig}
shows that
\begin{equation}\label{eq:2.4}
1 - a \text{ is invertible in } V_\ell^\infty (G,B)^+ .
\end{equation}
(See \cite[Theorem 5.12]{SolThesis} for an analogous argument in a different context.)
The same calculation entails that $a \mapsto (1 - a)^{-1}$ is continuous around 0 in 
$V_\ell^\infty (G,B)$.\\
(d) This follows from parts (a),(b),(c), \eqref{eq:2.4} and \cite[Lemma 17]{Vig}.
\end{proof}

Suppose that $K \subset G$ is a compact open subgroup such that $\ell$ is $K$-biinvariant. 
Let $e_K \in L^2 (G)$ be $\mu (K)^{-1}$ times the indicator function of $K$. This
is a projection in $C_r^* (G)$ and in the multiplier algebra of $C_r^* (G,B)$.
Right multiplication by $e_K$ just means averaging a measurable function $f : G \to B$ 
over $K$, making it right-$K$-invariant. Left multiplication of $f$ by $e_K$ 
can described explicitly as
\begin{equation}\label{eq:2.12}
(e_K * f)(g) = \mu (K)^{-1} \int_K k (f(k^{-1}g)) \textup{d}\mu (k) .
\end{equation}
If $f = e_K * f$, then we say that $f$ is twisted left-$K$-invariant. Equivalently,
$f(kg) = k (f(g))$ for all $g \in G, k \in K$.

For $r \in \R$ we define a norm on $C_c (G,B)$ by
\[
\nu_r (f) = \Big( \int_G \norm{f(g)}_B^2 \sigma (g)^{2r} \textup{d}\mu (g) \Big)^{1/2} .
\]
\texttt{Note: $\nu_0 (f)$ is not the norm of $f$ in the Hilbert $C^*$-module $L^2 (G,B)$,\\
that would be
\[
\norm{f}_{L^2 (G,B)} = \norm{  \int_G f(g) f(g)^* \textup{d} \mu (g) }_B^{1/2} 
\]
}
Let $S_\ell^\infty (G,B)$ be the completion of $C_c (G,B)$ with respect to the family
of norms $\nu_r \; (r \in \Z)$. Following \cite{Vig}, we call it the space of 
$\ell$-rapidly decreasing functions in $L^2 (G,B)$. See \cite{Sch} for many similar 
dense Fr\'echet subspaces of $L^2 (G,B)$.

Let $e_K S_\ell^\infty (G, B) e_K$ be the subspace of $S_\ell^\infty (G,B)$
consisting of right-$K$-invariant, twisted left $K$-invariant maps. Equivalently, 
$e_K S_\ell^\infty (G, B) e_K$ is the closure of\\ 
$e_K C_c (G, B) e_K$ with respect to the norms $\nu_r \; (r \in \Z)$.
Write $e_K C_r^* (G, B) e_K$ and $e_K V_\ell^\infty (G, B) e_K$ for the subalgebras of
right-$K$-invariant, twisted left-$K$-invariants element in, respectively, 
$C_r^* (G,B)$ and $V_\ell^\infty (G,B)$. Notice that $e_K$ is the identity element of
the multiplier algebra of $e_K C_r^* (G, B) e_K$.

\begin{lem}\label{lem:2.6}
\enuma{
\item $e_K V_\ell^\infty (G, B) e_K \subset e_K S_\ell^\infty (G, B) e_K$.

\texttt{Probably not true!}
\item $e_K V_\ell^\infty (G, B) e_K$ is closed under the holomorphic functional calculus of\\ 
$e_K C_r^* (G, B) e_K$.
} 
\end{lem}
\begin{proof}
(a) The unitization $B^+$ of $B$ is also a $G$-$C^*$-algebra, in a 
natural way. Theorem \ref{thm:2.4} also applies with $B^+$ instead of $B$. The identity 
element of $e_K C_r^* (G, B^+) e_K$ is $e_K$. 
For $a \in e_K V_\ell^\infty (G, B) e_K \subset e_K V_\ell^\infty (G, B^+) e_K$:
\[
D(\lambda (a)) (e_K) = \sigma (a * e_K) - a * (\sigma e_K) = \sigma a - a = \ell a .
\]
With induction we obtain $D^n (\lambda (a)) (e_K) = \ell^n a$. By assumption\\
$D^n (\lambda (a)) \in \mc B (L^2 (G,B^+))$, so 
\[
\ell^n a \in L^2 (G,B^+) \quad \forall n \in \Z_{\geq 0}.
\]
\texttt{The norm of $L^2 (G,B^+)$ differs from $\nu_0$, it is not clear 
whether the above implies that $\nu_r (a)$ is finite!}

This says that $\nu_r (a) < \infty$ for all $r \in \Z$, so 
\[
a \in S_\ell^\infty (G,B^+) \cap e_K V_\ell^\infty (G, B) e_K \subset
e_K S_\ell^\infty (G, B) e_K . 
\]
(b) Recall from \cite[\S A.1.5]{Bos} that the holomorphic functional calculi of\\
$(e_K V_\ell^\infty (G, B) e_K)^+$ and $(e_K C_r^* (G, B) e_K)^+$ can both be expressed as
\[
f(a) = (2 \pi i)^{-1} \int_\Gamma F(z) (z - a)^{-1} \textup{d} z ,
\]
where $\Gamma$ is a suitable contour around the spectrum of $a$, and $F$ is
a primitive of a holomorphic function $f$. From this expression we see that
it suffices to prove that every element of $(e_K V_\ell^\infty (G, B) e_K)^+$
which is invertible in $(e_K C_r^* (G, B) e_K)^+$, is already invertible in 
$(e_K V_\ell^\infty (G, B) e_K)^+$.

Let $a \in (e_K V_\ell^\infty (G, B ) e_K)^+ \cap \big( (e_K C_r^* (G, B) e_K)^+ \big)^\times$. 
Notice that $(e_K C_r^* (G, B) e_K)^+$ is naturally embedded in $e_K C_r^* (G, B^+) e_K$. 
We can identify its unit element with $e_K$. This has to be distinguished from the
unit element of $C_r^* (G,B^+)^+$, which we denote simply by 1. Then
$a + (1 - e_K)$ is invertible in $C_r^* (G,B^+)^+$, with inverse $a^{-1} + (1 - e_K)$.
By Theorem \ref{thm:2.4}.d
\[
a^{-1} + (1 - e_K) \in V_\ell^\infty (G,B^+)^+ .
\]
Then also 
\[
a^{-1} = e_K ( a^{-1} + 1 - e_K) \in V_\ell^\infty (G,B^+)^+ .
\]
At the same time $a^{-1} \in (e_K C_r^* (G, B ) e_K)^+$, so
\[
a^{-1} \in V_\ell^\infty (G,B^+)^+ \cap (e_K C_r^* (G, B ) e_K)^+ 
\subset (e_K V_\ell^\infty (G,B) e_K)^+ . \qedhere
\]
\end{proof}

From the density theorem for K-theory \cite[Th\'eor\`eme A.2.1]{Bos}, Theorem \ref{thm:2.4} 
and Lemma \ref{lem:2.6}.b we immediately conclude:

\begin{cor}\label{cor:3.1}
The Fr\'echet algebra homomorphisms 
\[
V_\ell^\infty (G, B) \to C_r^* (G, B) \text{ and }
e_K V_\ell^\infty (G, B) e_K \to e_K C_r^* (G, B) e_K
\]
induce isomorphisms on K-theory. 
\end{cor}

\section{Linear algebraic groups over non-archimedean local fields}
\label{sec:3}

In this paragraph $F$ is a non-archimedean local field and $\mc G$
is a connected reductive group defined over $F$. We endow $G = \mc G (F)$ with the
topology coming from the metric of $F$, making it into a locally compact totally
disconnected Hausdorff group. We denote the Bruhat--Tits building of $\mc G (F)$ by $X$. 
More generally our below arguments work for (possibly 
disconnected) quasi-reductive groups over non-archimedean local fields, by \cite{Sol2}.
But since every quasi-reductive group is embedded in a reductive group, nothing would
be gained by working in that generality. For background on the upcoming notions,
we refer to \cite{Tit}.

We fix a special vertex $x_0$ of $X$ and we let $G_{x_0}$ be its stabilizer in $G$. 
By the properness of the action, $G_{x_0}$ is compact. Because $G$ preserves 
the polysimplicial structure of $X$, the $G$-orbit of $x_0$ consists of vertices. 
Those lie discretely in $X$, so $G_{x_0}$ is open in $G$. We normalize the Haar
measure of $G$ so that $\mu (G_{x_0}) = 1$. 

We define
\[
\ell : G \to \R_{\geq 0}, \quad \ell (g) = d (g x_0, x_0 ).
\]
Since $G$ acts continuously and isometrically on $X$, this is a continuous length function 
and $\ell (g) = \ell (g^{-1})$. Notice that $\ell$ is biinvariant under $G_{x_0}$.

Let $S$ be a maximal $F$-split torus of $G$, such that $x_0$ lies in the apartment 
$\mh A_S$ of $X$ associated to $S$. Then $M := Z_G (S)$ is a minimal Levi subgroup 
of $G$. It has a unique maximal compact subgroup, namely $M_{\mr{cpt}} = M \cap G_{x_0}$. 
Then $M / M_{\mr{cpt}}$ can be identified with a lattice in the apartment $\mh A_S$.

Let $P = M U$ be a minimal parabolic subgroup of $G$, with unipotent radical $U$
and Levi factor $M$. (To be precise, one should say something like $\mc P$ is a
minimal parabolic $F$-subgroup of $\mc G$, and $G = \mc G (F), P = \mc P (F)$.) 
Recall the Iwasawa decomposition:
\begin{equation}\label{eq:2.3}
G = P G_{x_0}.  
\end{equation}
The torus $Z(M)^\circ$ acts algebraically on the Lie algebra of $U$, and that representation 
decomposes as a direct sum of algebraic characters $\chi : Z(M)^\circ \to F^\times$. 
Let $\norm{}_F$ denote the norm of $F$. As $Z(M)^\circ$ is cocompact in $M$, $\norm{\chi}_F$ 
extends uniquely to a character $M \to \R_{>0}$. We write
\[
M^+ = \{ m \in M : \norm{\chi (m)}_F \leq 1 \text{ for all } \chi 
\text{ which appear in Lie}(U) \} .
\]
The Cartan decomposition says that the natural map
\begin{equation}\label{eq:2.7}
M^+ / M_{\mr{cpt}} \to G_{x_0} \backslash G / G_{x_0} \quad \text{is bijective.} 
\end{equation}
Notice that $\delta_P (m) \leq 1$ for all $m \in M^+$, where $\delta_P : P \to \R_{>0}$ 
is the modular function. 
Using \eqref{eq:2.3} we extend $\delta_P$ to a right-$G_{x_0}$-invariant 
function on $G$. From \cite[\S II.1]{Wal} we recall that Harish-Chandra's $\Xi$-function
\[
\Xi (g) = \int_{G_{x_0}} \delta_P (kg)^{1/2} \textup{d}\mu (k) 
\]
is $G_{x_0}$-biinvariant. Recall from Proposition \ref{prop:3.5} that for $t \in \R_{> r_G}$ 
there exist unconditional completions $\mc S^t (G)$ and Banach algebras $\mc S_t (G,B)$ 

Let $K \subset G_{x_0}$ be an open subgroup. Then $K$ is also closed in the compact
Hausdorff group $G_{x_0}$, and hence compact. Since $\ell$ and $\Xi$ are
$G_{x_0}$-biinvariant, they descend to functions $K \backslash G / K \to \R$. For 
a right-$K$-invariant, twisted left-$K$-invariant measurable function $f : G \to B$,
\eqref{eq:2.12} and the $G$-invariance of the norm of $B$ imply that $\norm{f (g)}_B$ is 
$K$-biinvariant. Writing $\sigma = 1 + \ell$, we find that for such $f$:
\[
\norm{f}_{\mc S_t (G,B)} = \sup_{g \in \KGK} \norm{f(g)}_B 
\Xi (g)^{-1} \sigma (g)^t .
\]
Imposing (twisted) $K$-biinvariance enables us to fit functions in $\mc S_t (G,B)$:

\begin{lem}\label{lem:2.1}
For all $f \in e_K S_\ell^\infty (G,B) e_K$ and all $t > r_G$: 
$\norm{f}_{\mc S_t (G,B)} < \infty$.

In particular $\mc S_t (G,B)$ contains all elements of $V_\ell^\infty (G,B)$ 
that are right-invariant and twisted left-invariant under some compact open 
subgroup of $G_{x_0}$.
\end{lem}
\begin{proof}
Since $K$ is open and compact, the space $\KGK$ is discrete and its elements
have finite volume (from the measure on $G$). 
Choose a set of representatives $k_i \; (i = 1,\ldots ,[G_{x_0}:K])$ for $G_{x_0} / K$ 
and a set of representatives $m_j (j \in J)$ for $M^+ / M_{\mr{cpt}}$. By the Cartan 
decomposition \eqref{eq:2.7}, the natural map
\begin{equation}\label{eq:2.8}
\{ k_{i'}^{-1} m_j k_i : 1 \leq i,i' \leq [G_{x_0} : K], j \in J \} \to \KGK
\end{equation}
is surjective. The map 
\begin{equation}\label{eq:2.9}
M / M_{\mr{cpt}} \to X : m M_{\mr{cpt}} \mapsto m x_0
\end{equation}
sends $M / M_{\mr{cpt}}$ bijectively to a lattice in the apartment $\mh A_S$. 
Combining \eqref{eq:2.8} and \eqref{eq:2.9} with the definition of $\ell$, 
we deduce that $\KGK$ has polynomial growth with respect to $\ell$. Knowing that, 
\cite[Lemma 9]{Vig} says that the rapid decay of $f$ (in the $L^2$-sense) 
is equivalent to
\begin{equation}\label{eq:2.11}
\sup_{g \in \KGK} \norm{f(g)}_B \mu (KgK)^{1/2} \sigma (g)^t < \infty
\quad \text{for all } t \in \R .
\end{equation}
Notice that $\mu (KgK)$ and $\sigma (g)$ are $G_{x_0}$-biinvariant.
From \cite[p. 241]{Wal} we see that there exist $C_K, C'_K \in \R_{>0}$ such that
\begin{equation}\label{eq:2.5}
C_K \delta_P (m)^{-1} \leq \mu (K m K) \leq C'_K \delta_P (m)^{-1}
\quad \text{for all } m \in M^+.
\end{equation}
By \eqref{eq:2.8} and \eqref{eq:2.5}, the condition \eqref{eq:2.11} is equivalent to
\begin{equation}\label{eq:2.10}
\sup_{i,i',j} \norm{f(k_{i'}^{-1} m_j k_i)}_B \delta_P (m_j)^{-1/2} 
\sigma (m_j)^t < \infty \quad \text{for all } t \in \R .
\end{equation}
We recall from \cite[Lemma II.1.1]{Wal} that there exist $C_1,C_2 \in \R_{>0}$ and 
$d \in \N$ such that
\[
C_1 \delta_P (m)^{1/2} \leq \Xi (m) \leq C_2 \delta_P (m)^{1/2} \sigma (m)^d
\quad \text{for all } m \in M^+.
\]
With that, \eqref{eq:2.10} becomes equivalent to
\[
\sup_{i,i',j} \norm{f(k_{i'}^{-1} m_j k_i)}_B \Xi (m_j)^{-1} \sigma (m_j)^t < \infty
\quad \text{for all } t \in \R  .
\]
In view of the $K$-biinvariance of $\norm{f(g)}_B$ and $G_{x_0}$-biinvariance of the other 
involved terms, this says that 
\[
\norm{f}_{\mc S_t (G,B)} = \sup_{g \in G} \norm{f(g)}_B \Xi (g)^{-1} \sigma (g)^t 
\]
is finite. For the second claim we use Lemma \ref{lem:2.6}.a.
\end{proof}

In view of Corollary \ref{cor:3.1}, Lemma \ref{lem:2.1} and Proposition \ref{prop:3.5},
every class from\\
$K_* ( e_K C_r^* (G, B) e_K)$ can be represented by elements of matrix algebras over\\
$\mc S_t (G,B)^+$ (for $t > r_G$). This enables us to apply Corollary \ref{cor:1.5}
and to prove:

\begin{thm}\label{thm:3.2}
Let $F$ be a non-archimedean local field and let $\mc G$ be a connected reductive algebraic 
group defined over $F$. Endow $G = \mc G (F)$ with the topology coming from the metric 
of $F$. Let $B$ be a $\sigma$-unital $G$-$C^*$-algebra. The assembly map
\[
\mu^B : K_*^\top (G,B) \to K_* (C_r^* (G,B)) 
\]
is a bijection.
\end{thm}
\begin{proof}
As $G$ is totally disconnected and locally compact, its identity element admits a
neighborhood basis consisting of compact open subgroups $K$ \cite[\S 3.4.6]{Bou}. 
In particular $\bigcup_K e_K C_c (G , B) e_K$ is dense in $C_c (G,B)$. We partially
order these subgroups $K$ by reverse inclusion. Then 
\[
\{ e_K : K \subset G \text{ compact open subgroup} \}
\]
is an approximate identity consisting of projections in the multiplier algebra of
$C_r^* (G,B)$. Consequently
\begin{equation}\label{eq:3.3}
C_r^* (G,B) = \varinjlim_K e_K C_r^* (G,B) e_K  , 
\end{equation}
where the limits are taken in the category of $C^*$-algebras.
By the continuity of topological K-theory
\begin{equation}\label{eq:3.2} 
K_* (C_r^* (G,B)) = \varinjlim_K K_* ( e_K C_r^* (G, B) e_K) . 
\end{equation}
Pick any class $p \in K_* (C_r^* (G,B))$. By \eqref{eq:3.2} and Corollary \ref{cor:3.1}
it lies in the image of $K_* ( e_K V_\ell^\infty (G, B) e_K )$, for a suitable compact
open subgroup $K \subset G_{x_0}$. Then Lemma \ref{lem:2.1} and Proposition \ref{prop:3.5}.b
imply that $p$ can be represented by an element in a matrix algebra over the Banach
algebra $\mc S_t (G,B)^+$, for any $t > r_G$. In particular $p$ lies in the image of 
$K_* (\mc S_t (G,B)) \to K_* (C_r^* (G,B))$. 
This holds for arbitrary $p$, so $K_* (\mc S_t (G,B)) \to K_* (C_r^* (G,B))$ is
surjective. 

In Proposition \ref{prop:3.5}.a we saw that $\mc S_t (G)$ is an unconditional
completion of $C_c (G)$. As $\Xi (g^{-1}) = \Xi (g)$ \cite[Lemme II.1.4]{Wal},
\[
\norm{f^*}_{\mc S_t (G)} = \norm{f}_{\mc S_t (G)} \quad \text{for all } f \in C_c (G).
\]
Proposition \ref{prop:3.5}.c enables us to rescale this norm so that 
\[
\norm{f}_{C_r^* (G)} \leq \norm{f}_{\mc S_t (G)} \quad \text{for all } f \in C_c (G).
\]
Now we checked all the assumptions of Corollary \ref{cor:1.5}, so we can finally 
apply that result.
\end{proof}

Using the permanence properties of the Baum--Connes conjecture with coefficients
(discussed in the appendix), we can generalize Theorem \ref{thm:3.2} to larger classes of
groups.

\begin{thm}\label{thm:3.3}
Let $G$ be as in Theorem \ref{thm:3.2} and let $H$ be a closed subgroup of $G$.

Let $G'$ be a second countable, exact, locally compact group with an amenable closed 
normal subgroup $N$ such that $G'/N$ is isomorphic (as topological group) to $H$. 

Then $H$ and $G'$ satisfy BC with arbitrary coefficients. 
\end{thm}
\begin{proof}
Every separable $C^*$-algebra is $\sigma$-unital, so Theorem \ref{thm:3.2} says in 
particular that $G$ satisfies BC with arbitrary separable coefficients.
Apply Theorem \ref{thm:5.3} and to $G$ and $H$ to get the desired result for $H$.

Then apply Theorem \ref{thm:A.4} to $H,G'$ and $N$ to obtain the claim for $G'$.
\end{proof}

We note that Theorem \ref{thm:3.3} applies to every linear algebraic group over $F$,
because such a group can be embedded as a closed subgroup in $GL_n (F)$ for some
$n \in \N$.

\section{Linear algebraic groups over global fields}
\label{sec:4}

In this section we consider linear algebraic groups $\mc G$ defined over
a global field $k$. The points of $\mc G$ over the ring of adeles of $k$ form a locally compact 
group, usually called an adelic group.

BC (with trivial coefficients) for reductive adelic groups has been obtained by Baum, Millington 
and Plymen in \cite{BMP1}. Later Chabert, Echterhoff and Oyono-Oyono \cite[Theorem 0.7]{CEO} 
were able to show that all linear algebraic adelic groups over number fields satisfy BC. 
We will generalize these results to all linear algebraic groups and all global fields. 
Like the aforementioned work, our proofs rely on the following.

\begin{thm}\textup{\cite[Theorem 1.1]{BMP2}} \label{thm:4.2} \\
Let $G$ be a second countable locally compact group. Let $(G_n )_{n=1}^\infty$ be an increasing
sequence of open subgroups, with $\bigcup_{n=1}^\infty G_n = G$. Let $B$ be a $G$-$C^*$-algebra
and suppose that each $G_n$ satisfies BC with coefficients $B$. Then $G$ satisfies the 
Baum--Connes conjecture with coefficients $B$. 
\end{thm}

Let $\mb A_{k,\mr{fin}}$ be the ring of finite adeles of $k$, that is, the restricted product 
of the non-archimedean completions $k_v$. Let $(v_i )_{i=1}^n$ be an ordering of the finite 
places of $k$ and let $\mf o_{v_i}$ denote the ring of integers of $k_{v_i}$. Then 
$\mb A_{k,\mr{fin}}$ can be expressed as the increasing union of the open subrings
\begin{equation}\label{eq:4.1}
\mb A_{k,n} := k_{v_1} \times \cdots \times k_{v_n} \times \prod\nolimits_{i>n} \mf o_{v_i} .
\end{equation}
The ring of adeles $\mb A_k$ is the direct product $\mb A_{k,\mr{fin}} \times 
\prod_{v | \infty} k_v$, where the latter product runs over all infinite places of $k$. When
$k$ is a global function field, there are no infinite places and $\mb A_k = \mb A_{k,\mr{fin}}$.
On the other hand, every number field does possess infinite places (but only finitely many).
Like in \eqref{eq:4.1} we can write $\mb A_k$ as the increasing union of the open subrings
\begin{equation}\label{eq:4.2}
\mb A_{k,n} \times \prod_{v | \infty} k_v = k_{v_1} \times \cdots \times k_{v_n} \times 
\prod_{i>n} \mf o_{v_i} \times \prod_{k|\infty} k_v .
\end{equation}
Via the diagonal embedding, $k$ can be realized as a discrete cocompact subring of $\mb A_k$
\cite[Theorem IV.2.2]{Wei}.

\begin{thm}\label{thm:4.3}
Let $\mc G$ be a linear algebraic group defined over a global field $k$. 
\enuma{
\item $\mc G (\mb A_{k,\mr{fin}})$ satisfies BC with arbitrary coefficients.
\item Suppose that, for every infinite place $v$ of $k$, $\mc G (k_v)$ satisfies BC with
arbitrary separable coefficients (e.g. $\mc G (k_v)$ is compact or solvable or, more generally, 
amenable). Then the adelic group $\mc G (\mb A_k)$ satisfies BC with coefficients.
}
\end{thm}
\begin{proof}
The proof of part (a) is analogous to that of part (b) and slightly simpler, so we omit it.\\
(b) By \eqref{eq:4.1} $\mc{G}(\mb{A}_{k,\mr{fin}})$ is the increasing union of its open
subgroups $\mc G (\mb A_{k,n}) \times \prod_{v | \infty} \mc G (k_v)$. By Theorem \ref{thm:4.2} 
it suffices to establish the theorem for each of the subgroups 
\[
G_n := \mc G (\mb A_{k,n})  \times \prod_{v | \infty} \mc G (k_v) = 
\mc G (k_{v_1}) \times \cdots \times \mc G (k_{v_n}) \times 
\prod_{i > n} \mc G (\mf o_{v_i}) \times \prod_{v | \infty} \mc G (k_v) .
\]
By Theorem \ref{thm:3.3} each $\mc G (k_{v_i})$ satisfies BCC. By Tychonoff's
Theorem the product of compact groups $\prod\nolimits_{i > n} \mc G (\mf o_{v_i})$ 
is again compact. By \cite[Theorem 3.17.i]{ChEc2} BC with separable coefficients is inherited 
by finite direct products of groups, so $G_n$ satisfies BC with arbitrary separable coefficients. 
With Corollary \ref{cor:A.2} we can lift the separability requirement.
\end{proof}

With the permanence properties of BCC from the appendix, we can generalize Theorem \ref{thm:4.3}.
For $\mc G (k)$ the below was already stated in \cite[Theorem 1.3]{BMP2}.

\begin{cor}\label{cor:4.6}
Let $\mc G$ be a linear algebraic group defined over a global field $k$ and suppose 
that, for every infinite place $v$ of $k$, $\mc G (k_v)$ is amenable. Let $H$ be a closed
subgroup of $\mc G (\mb A_k)$, for instance $\mc G (k)$ with the discrete topology.

Let $G'$ be a second countable, exact, locally compact group with an amenable closed 
normal subgroup $N$ such that $G'/N$ is isomorphic (as topological group) to $H$. 

Then $H$ and $G'$ satisfy BC with arbitrary coefficients. 
\end{cor}
\begin{proof}
Apply Theorems \ref{thm:4.3} and \ref{thm:5.3} and to $G$ and $H$ to get the desired result 
for $H$. Since $k$ embeds discretely in $\mb A_k$, $\mc G (k)$ embeds in $\mc G (\mb A_k)$ 
as a discrete subgroup, and it is an example of such an $H$.

Then apply Theorem \ref{thm:A.4} to $H,G'$ and $N$ to obtain the claim for $G'$.
\end{proof}

Unfortunately part (b) of Theorem \ref{thm:4.3} does not apply to all linear algebraic 
groups over number fields, because the Baum--Connes conjecture with coefficients 
is still open for many reductive Lie groups.
A strong result in that direction was proven by Chabert, Echterhoff and Nest. We present a 
slightly simplified version:

\begin{thm}\label{thm:4.4} \textup{\cite[Theorem 1.2]{CEN}} \\
Let $G$ be a second countable locally compact group, such that the identity component $G^\circ$ 
is a Lie group and $G / G^\circ$ is compact. (For instance, $G$ can be a finite dimensional
Lie group with only finitely many components.) Then $G$ satisfies the twisted Baum--Connes 
conjecture.
\end{thm}

We generalize this to adelic groups.

\begin{thm}\label{thm:4.5}
Let $\mc G$ be a linear algebraic group defined over a global field $k$.   
Then the adelic group $\mc G (\mb A_k)$ satisfies the twisted Baum--Connes conjecture.
\end{thm}
\begin{proof}
Let $\mc K (H)$ be the algebra of compact operators on a separable Hilbert space $H$, and
suppose that it carries the structure of a $\mc G (\mb A_k)$-$C^*$-algebra. We have to show 
that $\mc G (\mb A_k)$ satisfies BC with coefficients $\mc K (H)$. 

Write $N = \prod_{v | \infty} \mc G (k_v)$. As the $\R$-points of an algebraic group, this is
a (finite dimensional) Lie group with only finitely many connected components. Suppose that
$L \subset \mc G (\mb A_k)$ is a closed subgroup containing $N$ as cocompact subgroup. Then
\[
L / N \subset \mc G (\mb A_k) / N \cong \mc G (\mb A_{k,\mr{fin}})
\]
is totally disconnected, and hence $L^\circ = N^\circ$. Consequentely $L / L^\circ = L / N^\circ$
is compact, and by Theorem \ref{thm:4.4} $L$ satisfies BC with coefficients $\mc K (H)$.

The above checks that the conditions of \cite[Theorem 2.1]{CEO} are fulfilled. The statement
of \cite[Theorem 2.1]{CEO} is: BC for $\mc G (\mb A_k)$ with coefficients $\mc K (H)$ is
equivalent to BC for $(\mc G (\mb A_k),N)$ with coefficients $C_r^* (N,\mc K (H))$. 
Here the action of $(\mc G (\mb A_k), N)$ on $C_r^* (N,\mc K (H))$ is twisted. By 
\cite[Theorem 1]{Ech} this twisted action is $\mc G (\mb A_k)$-equivariantly Morita 
equivalent to an ordinary action of $\mc G (\mb A_k) / N$ on another $C^*$-algebra, say $B$. 
Actually, since 
\[
\mc G (\mb A_k) = N \times \mc G (\mb A_{k,\mr{fin}})
\]
we may take $B = C_r^* (N,\mc K (H))$. Then BC for $(\mc G (\mb A_k),N)$ with coefficients \\
$C_r^* (N,\mc K (H))$ is equivalent to BC for $\mc G (\mb A_k)/N \cong \mc G (\mb A_{k,\mr{fin}})$ 
with coefficients $B$ \cite[Proposition 5.6]{ChEc1}. The latter holds by Theorem \ref{thm:4.3}.a.
\end{proof}

\appendix
\section{Permanence properties of Baum--Connes with coefficients}

For technical reasons, we prove some of the results in the body of our paper initially 
only for separable coefficient algebras. In this appendix we discuss how the Baum--Connes
conjecture with separable coefficients for an exact group $G$ implies BC for $G$
with coefficients in an arbitrary $G$-$C^*$-algebra. This is made possible by the
work of Chabert--Echterhoff \cite{ChEc2} on the continuity of the topological side
of BCC.

Let $G$ be a second countable, locally compact group and let $B$ be any
$G$-$C^*$-algebra. Let $\{ B_i : i \in I\}$ be the set of separable $G$-stable
sub-$C^*$-algebras of $B$, partially ordered by inclusion. The second countablity of
$G$ entails that every element of $B$ is contained in such a separable subalgebra $B_i$.
It follows that 
\begin{equation}\label{eq:A.1}
\lim_{i \in I} B_i \cong B \qquad \text{as } G-C^*\text{-algebras.}
\end{equation}

\begin{prop}\label{prop:A.1}
Let $G,B$ and the $B_i$ be as above. There is a natural isomorphism
\[
\lim_{i \in I} K_*^\top (G,B_i) \cong K_*^\top (G,B) . 
\]
\end{prop}
\begin{proof}
In \cite[Proposition 7.1]{ChEc2} this was proven when $B$ is separable and 
$\{ B_i : i \in I\}$ is an arbitrary inductive system of separable $G$-$C^*$-algebras
with direct limit $B$. We check that the arguments in \cite{ChEc2} also work when $B$
is not separable. 

In \cite{ChEc2} $K_*^\top (G,B)$ is exhibited as a direct limit of groups
$KK_*^G (C_0 (X),B_i)$, where $i \in I$ and $X$ runs through some collection of
proper $G$-spaces. The maps relating these KK-groups to the limit group are given by
Kasparov products
\begin{equation}\label{eq:A.2}
KK_*^G (C_0 (X'),C_0 (X)) \otimes_\Z KK_*^G (C_0 (X),B_i) \otimes_\Z KK_*^G (B_i,B)
\to KK_*^G (C_0 (X'),B) .
\end{equation}
The only involved element of $KK_*^G (B_i,B)$ is associated to the inclusion $B_i \to B$,
while the only relevant elements of $KK_*^G (C_0 (X'),C_0 (X))$ are those induced by
a continuous map $X \to X'$ and Bott elements in $KK_*^G (C_0 (X \otimes \R), C_0 (X))$
or $KK_*^G (C_0 (X'), C_0 (X' \otimes \R))$. As Chabert and Echterhoff observe, the 
associativity of the Kasparov product is needed to exchange the order of certain direct
limits. After that, their arguments do not use any properties of the separable coefficient
algebras $B_i$, they only involve various constructions with commutative $C^*$-algebras.

We point out that, by \cite[Theorems 2.11 and 2.14.5]{Kas}, the associativity of the 
Kasparov product in \eqref{eq:A.2} holds even if $B$ is not separable. With this in mind,
the entire proof of \cite[Proposition 7.1]{ChEc2} also applies to our possibly 
non-separable $G$-$C^*$-algebra $B$.
\end{proof}

As the continuity of the analytic side of Baum--Connes follows from exactness of the group,
Proposition \ref{prop:A.1} has the following consequence:

\begin{cor}\label{cor:A.2}
Let $G$ be a second countable, exact, locally compact group. Suppose that $G$
satisfies the Baum--Connes conjecture with coefficients in any separable $G$-$C^*$-algebra.
Then $G$ satisfies BCC, that is, BC with arbitrary (possibly non-separable) coefficients.
\end{cor}
\begin{proof}
The inclusion $B_i \to B$ and the naturality of the assembly map \cite[\S 9]{BCH} yield
a commutative diagram
\[
\begin{array}{ccc}
K_*^\top (G,B_i) & \xrightarrow{\mu^{B_i}} & K_* (C_r (G,B_i)) \\
\downarrow & & \downarrow \\
K_*^\top (G,B) & \xrightarrow{\mu^B} & K_* (C_r (G,B)) 
\end{array}
\]
From the exactness of $G$ and \eqref{eq:A.1} we get a natural isomorphism
\[
\lim_{i \in I} C_r^* (G,B_i) \cong C_r^* (G,B) . 
\]
Combine these with Proposition \ref{prop:A.1}. Using the assumption that every
$\mu^{B_i}$ is an isomorphism, we find that $\mu^B$ is an isomorphism as well.
\end{proof}

From \cite[Theorem 2.5]{ChEc2} and Corollary \ref{cor:A.2} we get:

\begin{thm}\label{thm:5.3}
Let $G$ be a second countable, exact, locally compact group and let $H$ be 
a closed subgroup of $G$. Suppose that $G$ satisfies the Baum--Connes conjecture with 
arbitrary separable coefficients. 
Then $H$ satisfies BCC, that is, BC with arbitrary coefficients.
\end{thm}

Chabert, Echterhoff and Oyono-Oyono \cite{CEO} proved a permanence property of BCC with 
respect to extensions, which we now generalize to arbitrary coefficient algebras.

\begin{thm}\label{thm:A.4}
Let $G$ be a second countable, exact, locally compact group and let $N$ be 
a closed normal amenable subgroup of $G$. Suppose that $G/N$ satisfies BC with arbitrary 
separable coefficients. Then $G$ satisfies BCC (with arbitrary coefficients).
\end{thm}
\begin{proof}
By Corollary \ref{cor:A.2} it suffices to prove BC for $G$ with coefficients in an
arbitrary separable $G$-$C^*$-algebra $B$.

Suppose that $L$ is an extension of $N$ by a compact group. Then $L$ inherits the
amenability of $N$, so by \cite{HiKa} it satisfies BC with arbitrary separable coefficients.
This shows that the assumptions of \cite[Theorem 2.1]{CEO} are satisfied by $(G,N)$.
As moreover $B$ is separable, we may apply that result. It says that the bijectivity of
the assembly map $\mu^B$ is equivalent to: 
\[
\text{the pair } (G,N) \text{ satisfies Baum--Connes with coefficients } C_r^* (N,B). 
\]
The statement involves a twisted action of $(G,N)$ on $C_r^* (N,B)$. Fortunately, by 
\cite[Theorem 1]{Ech} this twisted action is $G$-equivariantly Morita equivalent to an 
ordinary action of $G/N$ on another separable $C^*$-algebra, say $B'$. Then Baum--Connes 
for $(G,N)$ with coefficients $C_r^* (N,B)$ is equivalent to Baum--Connes for $G/N$ with 
coefficients $B'$ \cite[Proposition 5.6]{ChEc1}. That holds by assumption.
\end{proof}


\begin{thebibliography}{99}

\bibitem[BaCo]{BaCo} P.F. Baum, A. Connes,
"K-theory for discrete groups", pp. 1--20 in:
\emph{Operator algebras and applications Vol. I},
London Math. Soc. Lecture Note Ser. {\bf 135},
Cambridge University Press, Cambridge, 1988

\bibitem[BCH]{BCH} P.F. Baum, A. Connes, N. Higson, 
``Classifying space for proper actions and K-theory of group $C^*$-algebras",
pp. 240--291 in: \emph{$C^*$-algebras: 1943--1993. A fifty year celebration}, 
Contemp. Math. {\bf 167}, American Mathematical Society, Providence RI, 1994

\bibitem[BGW]{BGW} P.F. Baum, E. Guentner, R. Willett,
``Expanders, exact crossed products, and the Baum--Connes conjecture"
arXiv:1311.2343, 2015

\bibitem[BMP1]{BMP1} P. Baum, S. Millington, R. Plymen,
``A proof of the Baum--Connes conjecture for reductive adelic groups",
C.R. Acad. Sci. Paris {\bf 332} (2001), 195--200

\bibitem[BMP2]{BMP2} P. Baum, S. Millington, R. Plymen,
``Local-global principle for the Baum--Connes conjecture with coefficients",
K-theory {\bf 28} (2003), 1--18

\bibitem[Bos]{Bos} J.-B. Bost,
``Principe d'Oka, K-th\'eorie et syst\`emes dynamiques non commutatifs",
Invent. Math. {\bf 101} (1990), 261--333

\bibitem[Bou]{Bou} N. Bourbaki, \emph{\'El\'ements de math\'ematique Livre III. 
Topologie g\'en\'erale}, Hermann, Paris, 1960

\bibitem[BrTi]{BrTi1} F. Bruhat, J. Tits, ``Groupes r\'eductifs sur un corps local: I.
Donn\'ees radicielles valu\'ees",
Publ. Math. Inst. Hautes \'Etudes Sci. {\bf 41} (1972), 5--251

\bibitem[ChEc1]{ChEc1} J. Chabert, S. Echterhoff,
``Twisted equivariant $KK$-theory and the Baum--Connes conjecture for group extensions",
K-theory {\bf 23} (2001), 157--200

\bibitem[ChEc2]{ChEc2} J. Chabert, S. Echterhoff,
``Permanence properties of the Baum--Connes conjecture",
Documenta Math. {\bf 6} (2001), 127--183

\bibitem[CEN]{CEN} J. Chabert, S. Echterhoff, R. Nest,
``The Connes--Kasparov conjecture for almost connected groups and for linear $p$-adic groups",
Publ. Math. Inst. Hautes \'Etudes Sci. {\bf 97} (2003), 239--278

\bibitem[CEO]{CEO} J. Chabert, S. Echterhoff, H. Oyono-Oyono,
``Going-down functors, the K\"unneth formula, and the Baum--Connes conjecture",
Geom. funct. Anal. {\bf 14} (2004), 491--528   

\bibitem[Ech]{Ech} S. Echterhoff,
``Morita equivalent twisted actions and a new version of the Packer--Raeburn stabilization trick",
J. London Math. Soc. {\bf 50} (1994), 170--186

\bibitem[ELN]{ELN} S. Echterhoff, K. Li, R. Nest,
``The orbit method for the Baum--Connes conjecture for algebraic groups over local function fields",
J. Lie Theory {\bf 28} (2018), 323--341

\bibitem[HC]{HC} Harish-Chandra, ```Spherical functions on a semisimple Lie group I",
Amer. J. Math. {\bf 80} (1958), 241--310

\bibitem[HiKa]{HiKa} N. Higson, G.G. Kasparov,
``E-theory and KK-theory for groups which act properly and isometrically on Hilbert space",
Invent. Math. {\bf 144} (2001), 23--74

\bibitem[HLS]{HLS} N. Higson, V. Lafforgue, G. Skandalis,
``Counterexamples to the Baum--Connes conjecture",
Geom. funct. anal. {\bf 12} (2002), 330--354

\bibitem[Jul1]{Jul} P. Julg,
``K-th\'eorie \'equivariante et produits crois\'es",
C.R. Acad. Sci. Paris {\bf 292} (1981), 629--632

\bibitem[Jul2]{Jul2} P. Julg, 
``Travaux de N. Higson and G. Kasparov sur la conjecture de Baum-Connes",
Ast\'erisque {\bf 252} (1998), 151--183

\bibitem[Kas]{Kas} G.G. Kasparov,
``Equivariant K-theory and the Novikov conjecture",
Invent. Math. {\bf 91} (1988), 147--201

\bibitem[KaSk1]{KaSk} G.G. Kasparov, G. Skandalis,
``Groups acting on buildings, operator K-theory, and Novikov's conjecture",
K-Theory {\bf 4.4} (1991), 303--337

\bibitem[KaSk2]{KaSk2} G.G. Kasparov, G. Skandalis,
``Groups acting properly on ``bolic" spaces and the Novikov conjecture",
Annals Math. {\bf 158} (2003), 165--206

\bibitem[KiWa]{KiWa} E. Kirchberg, S. Wassermann,
``Exact groups and continuous bundles of $C^*$-algebras",
Math. Ann. {\bf 315.2} (1999), 169--203

\bibitem[Laf]{Laf} V. Lafforgue, ``K-th\'eorie bivariante pour les alg\`ebres de Banach
et conjecture de Baum--Connes", Invent. Math. {\bf 149.1} (2002), 1--95

\bibitem[PlRa]{PlRa} V.P. Platonov, A. Rapinchuk,
\emph{Algebraic groups and number theory},
Pure and Applied Mathematics {\bf 139}, Academic Press, 1994

\bibitem[Sch]{Sch} L.B. Schweitzer,
``Dense m-convex Fr\'echet subalgebras of operator algebra crossed products by Lie groups",
Int. J. Math. {\bf 4.4} (1993), 661--673

\bibitem[Sol1]{SolThesis} M. Solleveld,
\emph{Periodic cyclic homology of affine Hecke algebras},
PhD Thesis, Universiteit van Amsterdam, 2007, arXiv:0910.1606

\bibitem[Sol2]{Sol2} M. Solleveld,
``Pseudo-reductive and quasi-reductive groups over non-archimedean local fields",
J. Algebra {\bf 510} (2018), 331--392

\bibitem[Tit]{Tit} J. Tits, ``Reductive groups over local fields",
pp. 29--69 in: \emph{Automorphic forms, representations and L-functions Part I},
Proc. Sympos. Pure Math. {\bf 33}, American Mathematical Society, Providence RI, 1979

\bibitem[Val]{Val} A. Valette,
\emph{Introduction to the Baum-Connes conjecture},
Lectures in Mathematics ETH Z\"urich, Birkh\"auser Verlag, Basel, 2002

\bibitem[Vig]{Vig} M.-F. Vign\'eras
``On formal dimensions for reductive $p$-adic groups",
pp. 225--266 in: \emph{Festschrift in honor of I.I. Piatetski-Shapiro
on the occasion of his sixtieth birthday, Part I},
Israel Math. Conf. Proc. {\bf 2}, Weizmann, Jerusalem, 1990

\bibitem[Wal]{Wal} J.-L. Waldspurger,
``La formule de Plancherel pour les groupes $p$-adiques (d'apr\`es Harish-Chandra)",
J. Inst. Math. Jussieu {\bf 2.2} (2003), 235--333

\bibitem[Wei]{Wei} A. Weil,
\emph{Basic number theory 3rd ed.},
Grundlehren der mathematischen Wissenschaften {\bf 144}, Springer-Verlag, 1974

\end{thebibliography}
\end{document}